\documentclass[a4paper,11pt,reqno]{amsart}
\usepackage{amssymb}
\usepackage{enumitem}
\usepackage[alphabetic]{amsrefs}
\usepackage{tikz}
\usepackage{blindtext, xcolor}
\usepackage{comment}
\usepackage{mathrsfs}
\usepackage{geometry}
\usepackage{a4wide}
\usepackage{fullpage}
\usepackage{mathtools}

\makeatletter
\@namedef{subjclassname@2020}{%
  \textup{2020} Mathematics Subject Classification: 5C31 primary; 82B20, 68W25 secondary} 
\makeatother

\newtheorem{corollary}{Corollary}[section]
\newtheorem{theorem}[corollary]{Theorem}
\newtheorem{lemma}[corollary]{Lemma}
\newtheorem{proposition}[corollary]{Proposition}

\newtheorem{definition}[corollary]{Definition}
\newtheorem{question}[corollary]{Question}

\newtheorem{remark}[corollary]{Remark}
\newtheorem{example}[corollary]{Example}
\newtheorem*{theorem*}{Theorem}
\newtheorem*{corollary*}{Corollary}

\numberwithin{equation}{section}

\usepackage{graphicx} 

\title{A near-optimal zero-free disk for the Ising model}

\author[V. Patel]{Viresh Patel}
\address{Viresh Patel, School of Mathematical Sciences, Queen Mary University of London, Mile End Road, London, E1 4NS}
\email{\texttt{viresh.patel@qmul.ac.uk}}
\author[G.Regts]{Guus Regts}
\address{Guus Regts, Korteweg de Vries Institute for Mathematics, University of Amsterdam. P.O. Box 94248,  
1090 GE, Amsterdam, The Netherlands}
\email{\texttt{guusregts@gmail.com}}
\author[A. Stam]{Ayla Stam}
\address{Ayla Stam, Cygnus Gymnasium. Vrolikstraat 8, 1091 VG, Amsterdam, The Netherlands}
\email{a.stam@cygnusgymnasium.nl}
\thanks{GR was funded by the Netherlands Organisation of Scientific Research (NWO): VI.Vidi.193.068}
\date{\today}

\begin{document}

\begin{abstract}
The partition function of the Ising model of a graph $G=(V,E)$ is defined as $Z_{\text{Ising}}(G;b)=\sum_{\sigma:V\to \{0,1\}} b^{m(\sigma)}$, where $m(\sigma)$ denotes the number of edges $e=\{u,v\}$ such that $\sigma(u)=\sigma(v)$. 
We show that for any positive integer $\Delta$ and any graph $G$ of maximum degree at most $\Delta$, $Z_{\text{Ising}}(G;b)\neq 0$ for all $b\in \mathbb{C}$ satisfying $|\frac{b-1}{b+1}| \leq \frac{1-o_\Delta(1)}{\Delta-1}$ (where $o_\Delta(1) \to 0$ as $\Delta\to \infty$).
This is optimal in the sense that  $\tfrac{1-o_\Delta(1)}{\Delta-1}$ cannot be replaced by $\tfrac{c}{\Delta-1}$ for any constant $c > 1$ subject to a complexity theoretic assumption.

To prove our result we use a standard reformulation of the partition function of the Ising model as the generating function of even sets. We establish a zero-free disk for this generating function inspired by techniques from statistical physics on partition functions of polymer models.
Our approach is quite general and we discuss extensions of it to certain types of polymer models.

\end{abstract}

\maketitle
\section{Introduction}
The Ising model is an important and well-studied model of ferromagnetism in statistical physics. Its partition function encodes many physical parameters of the system and from the combinatorial perspective it encodes the generating function of cuts of the underlying graph. This paper concerns the location of the zeros of the Ising partition function, which are indicative of possible phase transitions in the model.

Let $G=(V,E)$ be a graph and let $b \in \mathbb{C}$ be a complex parameter called the edge interaction. A configuration $\sigma$ of the Ising model is a $2$-colouring of the vertices, that is a function $\sigma: V \rightarrow \{0,1\}$. The weight of $\sigma$ is given by $b^{m(\sigma)}$, where $m(\sigma)$ is the number of monochromatic edges of $\sigma$, i.e. the number of edges $\{u,v\} \in E$ satisfying $\sigma(u) = \sigma(v)$. The Ising partition function of $G$ with edge interaction $b$ is then given by
\begin{equation}\label{eq:def pf}
Z_{\rm Ising}(G; b) = \sum_{\sigma: V \rightarrow \{0,1\}} b^{m(\sigma)}.
   \end{equation}

Note that taking $b=e^{2\beta}$, where $\beta$ is the inverse temperature, gives the equivalent, more standard way of writing the partition function in statistical physics.

In statistical physics one is interested in the zeros of $Z_{\rm Ising}$ (also known as Fisher zeros~\cite{fisher}) because they indicate the possible presence of phase transitions in the model; here a phase transition occurs at a value of the parameter $b$ if small deviations in $b$ cause a large qualitative changes in the behaviour of the model. 
While our interest in the present paper is in these Fisher zeros, we note that one can also equip the model with an `external field', which transforms the partition function into a $2$-variable polynomial obtained from~\eqref{eq:def pf} by multiplying each term $b^{m(\sigma)}$ by the term $\lambda^{|\sigma^{-1}(0)|}$, for an external field like parameter $\lambda$. 
Zeros in the variable $\lambda$ of the partition function are called Lee-Yang zeros and are known to lie on the unit circle in the complex plane when $b>1$~\cite{LeeYang}; see~\cite{PR20ising} for a precise description. 

More recently, there has been much interest in these zeros also in theoretical computer science (for several models including the Ising and the hard-core model) because the transition from absence to presence of zeros often represents a computational phase transition, i.e.\ a transition in the computational tractability of estimating the partition function with fast algorithms.
We will not go into the details here, but refer the reader to \cites{fisher,FriedliVelenikbook} for the connection to statistical physics and to \cites{BarBook, PR17,LSSleeyang,LeeYang, BGPRhardness,GGHisingbounded,PRsurvey} for the connection to algorithms. 
There are also connections to quantum mechanics and quantum computing where the interest is in non-real values of $b$; see~\cites{mann,mann2023algorithmic} and the references therein for more details.


Graphs of bounded maximum degree form the natural setting for studying zeros of the Ising partition function. For each positive integer $\Delta$, we are interested in regions $R$ of the complex plane for which $Z_{\rm Ising}(G;b) \not= 0$ for all graphs $G$ of maximum degree at most $\Delta$ and all $b \in R$. Thus, we are interested in establishing the existence of regions $R$ that are contained in 
\[
\mathcal{R}_{\Delta} = \{b \in \mathbb{C}: Z_{\rm Ising}(G; b) \not= 0 \text{ for all } G \text{ of maximum degree } \Delta\}.
\]

\begin{itemize}
\item Liu, Sinclair, and Srivastava~\cite{LSSfisherzeros}*{Theorem 1.2} showed that for each $\Delta$ and $\varepsilon>0$ there is a thin strip $S=S_{\Delta,\varepsilon} \subseteq \mathbb{C}$ of the complex plane containing the interval $[\frac{\Delta - 2}{\Delta}+\varepsilon, \frac{\Delta}{\Delta -2}-\varepsilon]$ such that $S \subseteq \mathcal{R}_{\Delta}$. The height of this strip is positive and decreases as $\varepsilon$ decreases. 

\item Galanis, Goldberg, and Herrera-Poyatos~\cite{GGHisingbounded} recently showed that for each $\Delta$, the disc $D(\varepsilon_\Delta)$ is contained in $\mathcal{R}_{\Delta}$, where $\varepsilon_{\Delta} := \tan(\frac{\pi}{4(\Delta - 1)})$ and where for $r>0$, $D(r)$ is the disk defined as
\[
D(r):=\{b\in \mathbb{C}:\left|\tfrac{b-1}{b+1}\right|\leq r\}.
\]
We note that for large $\Delta$, $D(\varepsilon_\Delta)$ is the disc whose diameter is approximately the interval $[\frac{\Delta - \pi/4}{\Delta + \pi/4}, \frac{\Delta + \pi/4}{\Delta - \pi/4}]$.
\item Earlier, Barvinok~\cite{BarBook} with an extension by Mann and Bremner~\cite{mann} established (for each $\Delta$) the existence of a certain disc inside $\mathcal{R}_{\Delta}$. These disks are contained in $D(\varepsilon_\Delta)$ and so we do not detail them here (see e.g.~\cite{GGHisingbounded} for a description). 
Barvinok and Barvinok~\cite{BarBarIsing}*{Theorem 1.1} established the existence of a more complicated diamond-shaped region inside $\mathcal{R}_{\Delta}$. We refer to~\cite{GGHisingbounded}*{Figure 1} for an illustration of this for $\Delta=3$. 
It should be noted that the results discussed above in fact concern more general partition functions than that of the Ising model, with the former about graph homomorphism partition functions and the latter about the multivariate Ising model with an external field.
\end{itemize}    

Our main result is to give a new, near-optimal zero-free disk for $Z_{\rm Ising}$ that is larger than the regions described above. 
\begin{theorem}
\label{thm:main}
For each positive integer $\Delta \geq 3$ we have $Z_{\rm Ising}(G, b) \not= 0$ for every graph $G$ of maximum degree at most $\Delta$ and every $b \in D(n_\Delta)$, where 
\[
n_\Delta:=\frac{\left(1-\tfrac{1}{\sqrt{2(\Delta-1)}}\right)^2}{\Delta-1} = \frac{1 - o_{\Delta}(1)}{\Delta - 1}.
\]
Moreover, for any $\varepsilon\in (0,1)$ there exists $g=g_\Delta\in \mathbb{N}$ such that if $G$ additionally has girth at least $g$, then $Z_{\rm Ising}(G, b) \not= 0$ for any $b \in D(\tfrac{1-\varepsilon}{\Delta-1})$.
\end{theorem}

We make a few remarks. Comparing to previous results, our zero-free disc $D(n_\Delta)$ includes the previous zero-free disc $D(\varepsilon_\Delta)$, (for $\Delta$ large enough) and nearly contains the zero-free strip $S_{\Delta, \varepsilon}$ and the diamond-shaped region from~\cite{BarBarIsing}, but extends much further in the imaginary direction. 
Also, our result is optimal in the sense that we cannot replace $\frac{n_\Delta}{\Delta-1}$ with $\tfrac{c}{\Delta-1}$ for any $c > 1$ in Theorem~\ref{thm:main} unless \textsc{P $=$ NP}. Indeed if we could, then we would have a zero-free disc containing some real values  $0< b < \frac{\Delta -2 }{\Delta}$, and by Barvinok's interpolation method (see \cites{BarBook,PR17}), we would have an FPTAS (fully polynomial-time approximation scheme) for approximating $Z_{\rm Ising}(G,b)$ for graphs of maximum degree at most $\Delta$ and some $b < \frac{\Delta - 2}{\Delta}$; however such an algorithm does not exist unless \textsc{P $=$ NP} by work of~\cites{SlySun2spin,GSVinapprox}.

As alluded to above, a consequence of Theorem~\ref{thm:main} is that for each $b \in D(n_\Delta)$, there is an FPTAS to compute $Z_{\rm Ising}(G;b)$ for graphs of maximum degree at most $\Delta$. 
 On the other hand, it was shown by Galanis Goldberg and Herrera-Poyatos~\cite{GGHisingbounded} that it is \textsc{\#P}-hard to approximate $|Z_{\rm Ising}(G;b)|$ for graphs of maximum degree at most $\Delta$ when $b$ is a non-real algebraic number that satisfies $|\frac{b - 1}{b + 1}| \geq  1/ \sqrt{\Delta - 1}$ and $b\notin \{-i,i\}$.
Mann and Minko~\cite{mann2023algorithmic} build on this to show hardness of approximation whenever $b=e^{i \theta}$ and $|\theta|\geq \tfrac{6\pi}{5(\Delta-2)}$. (In fact, the result of~\cite{mann2023algorithmic} concerns multigraphs.)


In terms of the complement of $\mathcal{R}_{\Delta}$ i.e.\ the locations of zeros of $Z_{\rm Ising}$, our discussion above explains why computational hardness results such as~\cites{SlySun2spin,GSVinapprox, mann2023algorithmic} imply the existence of zeros under suitable complexity-theoretic assumptions such as \textsc{P $\not=$ NP} or \textsc{\#P $\not=$ FP}. In the absence of such assumptions, not much is known.
  Results from~\cite{SlySun2spin} strongly suggest (without any complexity-theoretic assumption) that for each $\Delta\geq 3$ there exists a sequence of $\Delta$-regular graphs $(G_i)$ and complex numbers $(b_i)$ such that
$Z_{\rm Ising}(G_i;b_i)=0$ and $b_i\to \tfrac{\Delta-2}{\Delta}$.\footnote{Indeed, they give a description of the free energy of the anti-ferromagnetic Ising model on the infinite $\Delta$-regular tree, which appears to be non-analytic at $b_c=\tfrac{\Delta-2}{\Delta}$, cf.~\cite{SlySun2spin}*{Figure 2(b)} (This figure is in fact for the hard-core model, but a similar picture is expected for the anti-ferromagnetic Ising model). Non-analyticity would imply that zeros of large girth $\Delta$-regular graphs accumulate at $b_c$ with the aid of some complex analysis, cf.~\cite{fisher}. It is not clear to us how to rigorously show that the free energy is indeed non-analytic at $b_c$.} 

Finally we discuss our methods.
In contrast to previous results mentioned above, the proof of Theorem~\ref{thm:main} uses the even subgraph representation of the Ising model due to Van der Waerden~\cite{vdw}. Our method is based on a novel form of the polymer method based on the block structure of subgraphs. 
The idea of possibly utilising the block structure in the context of the polymer method was initiated by Jackson and Sokal~\cite{JacksonSokal} who were aiming to prove bounds on the zeros of the chromatic polynomial. 
Here we give the first concrete application.
We believe this idea could have further applications, which we discuss at the end of the paper.

The paper is organised as follows. In the next section we give preliminaries on blocks and the even set generating function. In Section~\ref{sec:main} we prove Theorem~\ref{thm:main} via Theorem~\ref{thm:even zerofree}.  In Section~\ref{sec:extension} we generalise the main idea of our result to any partition function that behaves well on blocks, in the hope that it might find wider applications. 
This is based on the master's thesis of the third author~\cite{stam20}.
We finish with concluding remarks in Section~\ref{sec:conclusion}.
For the reader's convenience we have kept the paper self-contained.

\section{Preliminaries}
In this section we recall the connection between the even set generating function and the parition function of the Ising model and collect some graph theoretic definitions concerning blocks and so-called block paths that will be used throughout
For standard definitions in graph theory, see \cite{Diestel}.

\subsection{Even set generating function}
For a graph $G=(V,E)$ we say that $F\subseteq E$ is  \emph{even} (or an \emph{even set}) if each vertex in the spanning subgraph $(V,F)$ has even degree.
For a variable $x$ we define the \emph{even set generating function}, $Z_{\rm even}(G;x)$, of $G$ by
\begin{equation}\label{eq:def even set pf}
Z_{\rm even}(G;x):=\sum_{\substack{F\subseteq E\\ F{\rm even}}} x^{|F|}.
\end{equation}

It is well known that, after a change of variables, we can rewrite the Ising partition function as the even set generating function; this goes back to Van der Waerden~\cite{vdw}.

\begin{lemma}\label{lem:vdw}
Let $G=(V,E)$ be a graph and let $x\in\mathbb{C}\setminus\{1\}$. Then
\begin{equation}\label{eq:vdw}
{Z}_{\rm even}(G;x)=(1-x)^{|E|}2^{-|V|}Z_{\rm Ising}(G;\tfrac{1+x}{1-x}).
\end{equation}
\end{lemma}
For the reader's convenience we include a short combinatorial proof. We refer to~\cite{GrimmetJanson} for a short probabilistic proof.
\begin{proof}
We prove this in the more general case that $G$ is a multigraph (thus allowing multiple edges and loops) by induction on the number of edges.
In case $G$ has no edges the result is clearly true.
Next let $e$ be an edge of $G$. If $e$ is a loop, then  $Z_{\rm even}(G;x)=(1+x)Z_{\rm even}(G\setminus e ;x)$, while $Z_{\rm Ising}(G;b)=bZ_{\rm Ising}(G\setminus e ;b)$ implying by induction that~\eqref{eq:vdw} holds for $G$.
Next, if $e=\{u,v\}$ is not a loop, then we have
\begin{align}\label{eq:recursion even}
Z_{\rm even}(G;x)=xZ_{\rm even}(G/e;x)+(1-x)Z_{\rm even}(G\setminus e;x).
\end{align}
where $G/e$ denotes the (multi)graph obtained from $G$ by contracting $e$ (and retaining loops and multiple edges). To see this, note that any even set of $G$ either contains $e$ or not; even sets of the first kind correspond to even sets in $G/e$ that are not even sets in $G \setminus e$ (so contribute $x(Z_{\rm even}(G/e;x) - Z_{\rm even}(G\setminus e;x))$, while even sets of the second kind correspond to even sets of $G \setminus e$ (so contribute $Z_{\rm even}(G\setminus e;x)$).
We also have
\begin{align}
\label{eq:recursion Ising}
Z_{\rm Ising}(G\setminus e;b)=Z_{\rm Ising}(G;b)+(1-b)Z_{\rm Ising}(G/e;b),
\end{align}
since $2$-colourings $\sigma:V(G\setminus e)\to \{0,1\}$ of $G \setminus e$ in which $u,v$ receive different colours correspond to $2$-colourings of $G$ in which $u,v$ receive different colours;  $2$-colourings of $G \setminus e$ in which $u,v$ receive the same colour correspond to $2$-colourings of $G/e$, and $2$-colourings of $G$ in which $u,v$ receive the same colour correspond to $2$-colourings of $G/e$ with an extra factor $b$.
From \eqref{eq:recursion even}, \eqref{eq:recursion Ising}, and induction, we have that $(1-x)^{|E(G)|}2^{-|V(G)|}Z_{\rm Ising}(G;\tfrac{1+x}{1-x})$ equals
\begin{align*}
&(1-x)^{|E(G)|}2^{-|V(G)|}\left(Z_{\rm Ising}(G\setminus e;\tfrac{1+x}{1-x})-(1-\tfrac{1+x}{1-x})Z_{\rm Ising}(G/e;\tfrac{1+x}{1-x}) \right)
\\
=&(1-x)Z_{\rm even}(G\setminus e;x)-\tfrac{1-x}{2}(1-\tfrac{1+x}{1-x})Z_{\rm even}(G/e;x)
\\
=&(1-x)Z_{\rm even}(G\setminus e;x)+xZ_{\rm even}(G/e;x),
\end{align*}
which is equal to $Z_{\rm even}(G;x)$ by~\eqref{eq:recursion even}.
\end{proof}

\begin{remark}\label{rem:convert}
  The lemma implies that if $Z_{\rm even}(G;x)$ is nonzero for all $x$ contained in the disk of radius $r<1$ centered at $0$, then  $Z_{\rm Ising}(G;b)$ is non zero for all $b$ satisfying $\left| \frac{b-1}{b+1}\right|\leq r$, i.e. the disk with diameter given by the real interval $[\frac{1-r}{1+r}, \frac{1+r}{1-r}]$.
\end{remark}

\subsection{Blocks and block paths}
Let $G=(V,E)$ be a connected graph. For $U\subseteq V$,  we denote by $E(U)$ the collection of edges $ab \in E$ such that  $a,b \in U$.
For a set of edges $A\subseteq E$, we denote by $V(A)$ the collection of vertices that are contained in some edge in $A$.

Recall that a graph $G$ is $2$-connected if $G$ is connected and $G-v$ is connected for all $v \in V(G)$. 
We work throughout with the convention that the edge $K_2$ is a $2$-connected graph (some authors require that a $2$-connected graph should have at least $3$ vertices).
A \emph{block} of a graph $G$ is a maximal $2$-connected subgraph of $G$ (and is therefore necessarily an induced subgraph of $G$). Throughout, we identify blocks $B$ with their edge sets and write $V(B)$ for the vertices of $B$.
We mention here some standard properties of blocks; see e.g.\ \cite{Diestel}.
Let $G=(V,E)$ be a graph, let $B_1, \ldots, B_k$ be its blocks, and $v_1, \ldots, v_{\ell}$ be its cut vertices. 

\begin{itemize}
    \item[(P1)] The blocks of $G$ decompose the edges of $G$, i.e. each edge of $G$ belongs to a unique block of $G$.  Furthermore, any two blocks of $G$ either have an empty vertex intersection or they intersect in a single cut vertex of $G$.
    \item[(P2)] The block-cutpoint graph of $G$ is the graph with vertices $\{B_1, \ldots, B_k, v_1, \ldots, v_{\ell}\}$ and edges of the form $\{B_i,v_j\}$ where $v_i \in B_j$. The block-cutpoint graph of $G$ is a forest (and is a tree if $G$ is connected).
\end{itemize}
We call a block $B$ of a graph $G$ a \emph{leaf block} if it contains a single cut vertex of $G$; equivalently $B$ is a leaf block of $G$ if it corresponds to a leaf in the block-cutpoint graph of $G$.

The following fact relates even sets and blocks; it is a simple exercise but we include the proof for completeness. 
\begin{proposition}\label{lem: block path even}
Let $G=(V,E)$ be a graph. Then $F\subseteq E$ is even if and only if each block of $F$ is even.    
\end{proposition}
\begin{proof}
By considering each component of $F$ separately, we may assume $F$ is connected.
If all blocks of $F$ are even then $F$ is even since blocks are edge-disjoint (by P1). 
The converse follows by induction on the number of blocks of $F$. If $F$ is a block then there is nothing to prove; otherwise let $B \subseteq F$ be a leaf block of $F$, which exists by (P2). Note that the blocks of $F$ consist of $B$ together with the blocks of $F-B$ and that $V(B) \cap V(F-B) = \{ v \}$ for some $v \in V$. This means all vertices of $B$ except possibly $v$ have even degree (in $B$), and so $v$ also has even degree in $B$ by the handshake lemma; hence $B$ is even. This means $F-B$ is even and by induction all its blocks are also even, proving the lemma.
\end{proof}

An important notion that will allow us to determine an effective recursion for the even set generating function in the next section is that of a \emph{block path} due to Sokal and Jackson~\cite{JacksonSokal}. 
\begin{definition}[Block path]\label{def:block path}
Let $G=(V,E)$ be a graph with $U\subset V$ and $v\in V \setminus U$. A \emph{block path from $v$ to $U$} is a subgraph $H$ of $G$ satisfying
\begin{enumerate}
\item $V(H)\cap U$ consists of a single vertex $u$;
\item $H$ is connected;
\item neither $u$ nor $v$ are cut vertices of $H$;
\item either $H$ is $2$-connected, or it contains exactly two leaf blocks $B_1$ and $B_2$ with $u\in B_1$ and $v\in B_2$.
\end{enumerate}
We denote the collection of all such block paths $H$ by $\mathcal{BP}(v,U,G)$ and we simply write $\mathcal{BP}(v,u,G)$ if $U = \{ u \}$. Again, we generally identify each block path with its edge set. 
\end{definition}

Next we give some basic properties of block paths that we will need. First, we note below that block paths are precisely the subgraphs whose block-cutpoint graphs are paths.

\begin{proposition}
\label{pr:equiv}
    With the definition of block paths above, we have that $H \in \mathcal{BP}(v,U,G)$ if and only if all of the following hold:
    \begin{itemize}
        \item[(a)] the block-cutpoint graph of $H$ is a path $P = B_1v_1B_2v_2 \cdots B_{k-1}v_{k-1}B_k$;
        \item[(b)] $V(B_1) \cap U = \{ u \}$ for some $u \in U$ and $v \in V(B_k)$; and
        \item[(c)] $V(B_i) \cap U = \emptyset$ for $i = 2, \ldots, k$ and $V(B_i) \cap \{v\} = \emptyset$ for $i = 1, \ldots, k-1$.
    \end{itemize} 
\end{proposition}
\begin{proof}
    Assuming (a), (b), and (c) hold, then clearly (1) and (2) hold in the definition of block paths. If $u$ is a cut vertex of $H$, it appears in $P$ and therefore appears in two of the blocks $B_1, \ldots, B_k$ contradicting (c), and similarly for $v$; hence (3) holds. If $B_1 = B_k$, i.e.\ $P$ is a single vertex, then $H$ is a block and so $2$-connected; otherwise $H$ has two leaf blocks, namely $B_1$ and $B_k$, and so satisfies (4).

    Conversely, suppose $H \subseteq G$ satisfies (1)-(4) of Definition~\ref{def:block path}. By (4), the block-cutpoint graph of $H$ must be a path, proving (a), and furthermore, by (4), $u \in V(B_1)$ and $v \in V(B_k)$. By (1) and (3) we deduce $V(B_1) \cap U = \{u\}$ so (b) holds. By (3), $u$ and $v$ are not cut vertices, so they can only appear in one block; hence (c) holds.
\end{proof}

\begin{proposition}
    Suppose $H \in \mathcal{BP}(v,u,G)$ for some vertices $u,v \in V(G)$. For any $x \in V(G)$, there is a path $Q$ in $G$ from $u$ to $v$ that contains $x$.
\end{proposition}
\begin{proof}
    Using the previous proposition, we know the block-cutpoint graph of $H$ is a path $P = B_1v_1B_2v_2 \cdots B_{k-1}v_{k-1}B_k$ with $u \in V(B_1)$ and $v \in V(B_k)$. Assume first that $x$ is not a cut vertex and that $x \in V(B_r)$. Set $v_0 = u$ and $v_k = v$. In each block $B_i$, we can find a path $Q_i$ from $v_{i-1}$ to $v_i$ and in the block $B_r$ we take a path $Q_r$ from $v_{r-1}$ to $v_r$ through $x$ (which is possible since $B_r$ is $2$-connected). Concatenating these paths gives the desired path $Q$. We can similarly find the path $Q$ if $x$ is a cut vertex (which includes the case when $B_r$ is a single edge).
\end{proof}
The following technical fact will be required in the proof of Lemma~\ref{lem:decomposition} below.
\begin{proposition}
\label{pr:tech}
    Suppose a graph $G$ can be written as the edge-disjoint union of graphs in two different ways: $G = H_1 \cup H_1' = H_2 \cup H_2'$. Suppose further that $H_1, H_2 \in \mathcal{BP}(v,u,G)$, where $u,v \in V(G)$. Then for some $i \in \{ 1,2 \}$, $H_i'$ contains a connected component $C$ with $|V(H_i) \cap V(C)| \geq 2$.  
\end{proposition}
\begin{proof}
    If $V(H_1) = V(H_2)$ then since $H_1$ and $H_2$ are different, assume without loss of generality that we have an edge $e = xy \in E(H_1) \setminus E(H_2)$. So $e \in E(H_2')$ and so $x$ and $y$ belong to the same connected component $C$ of $H_2'$. Moreover $x,y \in V(H_2) = V(H_1)$ as required.

    If $V(H_1) \not= V(H_2)$, assume without loss of generality that $x \in V(H_1) \setminus V(H_2)$. By the previous proposition, we know there is a path $Q$ in $H_1$ from $u$ to $v$ containing $x$. Since $u,v \in V(H_2)$, there is a subpath $Q'$ of $Q$ that starts and ends in $V(H_2)$ but with all internal vertices (including $x$) outside $V(H_2)$. Therefore $Q'$ is part of some component $C$ of $V(H_2')$, and $C$ intersects $V(H_2)$ in at least the two endpoints of $Q'$.
\end{proof}

\section{Proof of the main theorem}
\label{sec:main}
In this section we prove Theorem~\ref{thm:main}. 
We will require the following generalisation of the even set generating function.

Let $G=(V,E)$ be a graph. For $U\subset V$, we define $\mathcal{E}(G\mid U)$ to be the collection of even sets $F$ of $G$ such that each component of $(V,F)$ contains at most one vertex of $U$.
We define the associated generating function by
\begin{equation}\label{eq:def even set pf extended}
Z_{\text{even}}(G\mid U;x):=\sum_{F\in \mathcal{E}(G\mid U)} x^{|F|}.
\end{equation}
Observe that for any vertex $u\in V$ we have $Z_{\text{even}}(G\mid \{u\};x)=Z_{\text{even}}(G;x)$.

The following lemma establishes the main recursion we require in the induction step of our main theorem.
\begin{lemma}\label{lem:decomposition}
Let $G=(V,E)$ be a graph and let $U\subset V$.
Then for any vertex $v\in V\setminus U$
\begin{equation}
Z_{\text{even}}(G\mid U;x)=Z_{\text{even}}(G\mid U\cup\{v\};x)+\sum_{\substack{B\in \mathcal{BP}(v,U,G)\\ B \text{ even}}}x^{|B|}Z_{\text{even}}(G\mid U\cup V(B);x).   
\end{equation}
\end{lemma}
\begin{proof}
It suffices to prove that any $F\in \mathcal{E}(G\mid U)\setminus  \mathcal{E}(G\mid U\cup\{v\})$ can be written uniquely as an edge-disjoint union $F = B \cup F'$, where $B\in \mathcal{BP}(v,U,G)$ is even and $F'\in  \mathcal{E}(G\mid U\cup V(B))$.

To prove this, note that by definition, any $F\in \mathcal{E}(G\mid U)\setminus  \mathcal{E}(G\mid U\cup\{v\})$ must contain a (nontrivial) component $C$ that contains the vertex $v$ and a unique vertex $u$ from $U$.
The block-cutpoint graph of $C$ is a tree and contains a path $P$ between the blocks containing $v$ and $u$ respectively. By shortening the path if necessary we can further ensure that $u$ and $v$ only occur in the leaf blocks of $P$, so that taking the union of the blocks that appear in $P$ gives us a block path $B$ from $v$ to $u$ in $G$ (by Proposition~\ref{pr:equiv}).
Furthermore, $B$ is even by Proposition~\ref{lem: block path even}.
Because the block-cutpoint graph of $C$ is a tree, any component of $C\setminus B$ intersects $V(B)$ in at most one vertex and does not intersect $U \setminus V(B)$.
Any other nontrivial component $C'$ of $F$ intersects $U$ in at most one vertex by construction and does not intersect $V(B)$. 
Define $F'=F\setminus B$ and note that by Proposition~\ref{lem: block path even} each component of $F'$ is even.
This gives the desired pair $(B,F')$.

For the uniqueness, note that we cannot have a second decomposition $F = B_1 \cup F_1'$ with $B_1\in \mathcal{BP}(v,U,G)$ and $F_1'\in  \mathcal{E}(G\mid U\cup V(B_1))$ since by Proposition~\ref{pr:tech} we may assume that $F_1'$ has a component that intersects $B_1'$ in at least two vertices, contradicting $F_1'\in  \mathcal{E}(G\mid U\cup V(B_1))$.
%
\end{proof}

 For a graph $G$ and a vertex $v$ of $G$, define $\mathcal{W}_{v}(G)$ to be the collection of walks from $v$ to itself in $G$ that use each edge at most once. Let $W_{G,v}(x)$ denote the associated generating function, that is,
\[
W_{G,v}(x)=\sum_{\substack{W\in \mathcal{W}_{v}(G)}}x^{|W|},
\]
where $|W|$ denotes the number of edges in the walk $W$.
The next lemma bounds the walk generating function and is used in our proof of Theorem~\ref{thm:even zerofree} below.
\begin{lemma}\label{lem:walk bound}
Let $\Delta, g\geq 3$ be integers, let $G=(V,E)$ be a graph of maximum degree at most $\Delta$ and girth at least $g$, and let $v$ be a vertex of $G$.
Then for any $c\in [0,1)$, we have 
\[W_{G,v}(\tfrac{c}{\Delta-1}) \leq \tfrac{\Delta c^g}{(\Delta - 1)^2 (1-c)}.\]
\end{lemma}
\begin{proof}
The number of walks in $G$ that start and end at $v$ with $k$ edges is at most $\Delta(\Delta-1)^{k-2}$ since there are $\Delta$ choices for the first edge and $\Delta -1$ choices for each subsequent edge except the last, which is forced.  Also each such walk has at least $g$ edges. Therefore 
\[
W_{G,v;G}(\tfrac{c}{ \Delta - 1}) \leq \sum_{k \geq g}\Delta (\Delta-1)^{k-2}(\tfrac{c}{\Delta - 1})^k = \tfrac{\Delta c^g}{(\Delta - 1)^2 (1- c)}.
\]
\end{proof}

By Lemma~\ref{lem:vdw} and Remark~\ref{rem:convert} the following theorem immediately implies the `moreover' part of our main result, Theorem~\ref{thm:main}.

\begin{theorem}\label{thm:even zerofree}
Let $\varepsilon \in (0,1)$ and let $\Delta, g \geq 3$ be integers satisfying 
\begin{equation}\label{eq:require}
g+2 \geq  \frac{\log(2\varepsilon^2 (\Delta - 1)^2/\Delta)}{\log(1- \varepsilon)}.
\end{equation}
For every graph $G=(V,E)$ of maximum degree at most $\Delta$ and girth at least $g$ and every $x\in \mathbb{C}$ satisfying $|x|\leq \tfrac{(1-\varepsilon)^2}{\Delta-1}$, we have  
$Z_{\text{even}}(G;x)\neq 0$. 
\end{theorem}

\begin{proof}
We will prove the following statement which immediately implies the theorem.
Fix $a = \varepsilon$ and $c= 1- \varepsilon$ with $\varepsilon \in (0,1)$ as in the statement of the theorem.
For every graph $G=(V,E)$ of maximum degree at most $\Delta$ and girth at least $g$, every $U\subset V$, every $v\in V\setminus U$, and every $x \in \mathbb{C}$ with $|x|\leq \tfrac{c(1-a)}{\Delta - 1}$, we have
\begin{itemize}
    \item[(i)] $Z_{\text{even}}(G\mid U;x)\neq 0$, and
    \item[(ii)] \[
    \left|\frac{Z_{\text{even}}(G\mid U;x)}{Z_{\text{even}}(G\mid U\cup \{v\};x)}-1\right|\leq a. 
    \]
\end{itemize}
We note that (ii) directly implies (i) and so it suffices to show (ii). We will do this by induction on $|V\setminus U|$.
In case $V=U$ we have $Z_{\text{even}}(G\mid U;x)=1$, showing (i), while (ii) is vacuous.
Next, let us assume that $U\subset V$ and $v\in V\setminus U$.
By induction we know that $Z_{\text{even}}(G\mid U\cup \{v\};x)\neq 0$ and so by Lemma~\ref{lem:decomposition} we have
\[
\frac{Z_{\text{even}}(G\mid U;x)}{Z_{\text{even}}(G\mid U\cup \{v\};x)}-1=\sum_{\substack{B\in \mathcal{BP}(v,U,G)\\ B \text{ even}}}x^{|B|}\frac{Z_{\text{even}}(G\mid U\cup V(B);x)}{Z_{\text{even}}(G\mid U\cup \{v\};x)}.
\]

Given some $B\in \mathcal{BP}(v,U,G)$ as in the sum above, write $V(B) = \{ v, b_1, \ldots b_k,u \}$, where $\{u\}=V(B)\cap U$ and let $B_i := \{v, b_1, \ldots, b_i\}$ with $B_0 := \{v\}$.
By using (ii) inductively in the telescoping product below we obtain that 
\[
\left| \frac{Z_{\text{even}}(G\mid U\cup V(B);x)}{Z_{\text{even}}(G\mid U\cup \{v\};x)}\right| 
= \prod_{i=1}^k \left| \frac{Z_{\text{even}}(G\mid U\cup B_i;x)}{Z_{\text{even}}(G\mid U\cup B_{i-1};x)}\right| 
\leq \left(\frac{1}{1-a}\right)^{|V(B)|-2}.
\]
Note that any even block path $B$ must have at least as many edges as vertices. Therefore we can replace the exponent by $|B|-2$.
We thus obtain,
\begin{align}
\left|\frac{Z_{\text{even}}(G\mid U;x)}{Z_{\text{even}}(G\mid U\cup \{v\};x)}-1\right|\leq& (1-a)^2\sum_{\substack{B\in \mathcal{BP}(v,U,G)\\ B \text{ even}}}\left(\frac{|x|}{1-a}\right)^{|B|} \nonumber
\\
\leq &(1-a)^2\sum_{\substack{B\in \mathcal{BP}(v,U,G)\\ B \text{ even}}}\left(\frac{c}{\Delta-1}\right)^{|B|}.\label{eq:bp bound}
\end{align}
Any even block path from $v$ to $U$ can be obtained in at least two ways as a closed Eulerian walk in $G$ starting and ending at $v$.  Hence we can bound the right-hand side of~\eqref{eq:bp bound} by half the walk generating function, that is, by
\[
\tfrac{1}{2}(1-a)^2 W_{G,v;g}\left(\tfrac{c}{\Delta-1}\right) 
\leq \tfrac{1}{2}(1-a)^2 \frac{\Delta c^g}{(\Delta - 1)^2(1 - c)}
\]
where the inequality follows by the previous lemma. To complete the induction, it is enough to show that the above expression is bounded above by $a$. 
Recall that $a = 1-c = \varepsilon$; thus the inequality we require is
\[
(1- \varepsilon)^{g+2} \leq \frac{2(\Delta - 1)^2\varepsilon^2}{\Delta},    
\]
which holds provided $g+2 \geq \frac{\log(2\varepsilon^2 (\Delta - 1)^2/\Delta)}{\log(1- \varepsilon)}$, as required.
\end{proof}

We can apply the theorem to the collection of all graphs of maximum degree at most $\Delta$ to obtain the following, which proves the first part of Theorem~\ref{thm:main} by Remark~\ref{rem:convert}.
\begin{corollary}\label{cor:zero-free}
Let $\Delta\geq 3$. Then for any graph $G$ of maximum degree at most $\Delta$ and $x\in \mathbb{C}$ such that $|x|\leq (1-\tfrac{1}{\sqrt{2(\Delta - 1})})^2/(\Delta-1)$, $Z_{\rm even}(G;x)\neq 0$.
\end{corollary}
\begin{proof}
By the previous theorem it suffices to verify~\eqref{eq:require} with $\varepsilon=\tfrac{1}{\sqrt{2(\Delta-1)}}$ and $g=3$.
Substituting this value of $\varepsilon$ into~\eqref{eq:require} and rearranging, we see that we need 
\[
\left(1- \tfrac{1}{\sqrt{2(\Delta-1)}}\right)^5\leq \frac{\Delta-1}{\Delta},
\]
which holds if $\Delta\geq 2$ (indeed the inequality already holds without the power of $5$ on the left hand side).
\end{proof}

\section{An extension: block polynomials}
\label{sec:extension}
In this section we indicate how the method used to prove  Theorem~\ref{thm:main}
generalises to a much larger class of generating functions and discuss how this relates to the polymer method.

Let $w$ be a graph invariant taking values in $\mathbb{C}$, i.e. $w$ is a function from the collection of all graphs to the complex numbers such that it assigns the same value to isomorphic graphs.
We call $w$ \emph{multiplicative} if $w(H_1\cup H_2)=w(H_1)\cdot w(H_2)$ for any two graphs $H_1,H_2$. Here $H_1\cup H_2$ denotes the disjoint union of the graphs $H_1$ and $H_2$.
We call $w$ \emph{$1$-multiplicative} if
\begin{itemize}
\item $w(K_1)=1$, where $K_1$ denotes the graph consisting of a single vertex,
\item $w$ is multiplicative, and
\item $w(H)=w(H_1)\cdot w(H_2)$ whenever $H$ is the union of two graphs $H_1$ and $H_2$ that have exactly one vertex in common. 
 \end{itemize}
 Note that if $w$ is $1$-multiplicative, then $w(H)=\prod_{B \text{ block of }H} w(B)$ by a simple induction argument.
\begin{example}[$1$-multiplicative graph invariants]\label{ex:block}  \quad 
    \begin{itemize}
        \item[(i)] Define for $x\in \mathbb{C}$ and a graph $H$, $w(H)=x^{|E(H)|}$ if $H$ is even and $w(H)=0$ otherwise. Then $w$ is $1$-multiplicative (by making use of Proposition~\ref{lem: block path even}).
        \item[(ii)] Any evaluation of the Tutte polynomial is $1$-multiplicative. Here we take the Tutte polynomial of a graph $G=(V,E)$ as
        \[
T(G;x,y)=\sum_{F \subseteq E}(x-1)^{k(F)-k(E)}(y-1)^{|F|+k(F)-|V|},       
        \]
        where $k(F)$ denotes the number of components of a graph $(V,F)$.
        \item[(iii)]  The homomorphism density into a vertex transitive graph is $1$-multiplicative. We recall that for graphs $H=(V,E)$ and $G=([k],F)$, the homomorphism density of $H$ into $G$, denoted by $t(H,G)$, is defined as
        \[
        t(H,G)=\frac{\sum_{\phi:V\to [k]}\prod_{uv\in E} \mathbf{1}_{\{\phi(u)\phi(v)\in F\}}}{k^{|V|}}.
        \]
          \end{itemize}
\end{example}

Given a $1$-multiplicative graph invariant $w$ and a graph $G=(V,E)$, define the \emph{block polynomial} as
\begin{equation}\label{eq:def blockpol}
Z_{\text{block}}(G;w):=\sum_{H\subseteq E}w(H)=\sum_{H\subseteq E}\prod_{B\text{ block of } H}w(B).   
\end{equation}
By Example~\ref{ex:block} (i), the even set generating function is an example of a block polynomial.

As in the previous section,  for each set $U\subset V$, define
\begin{equation}\label{eq:def blockpol extend}
Z_{\text{block}}(G\mid U;w):=\sum_{\substack{H\subseteq E\setminus E(U)}}w(H),
\end{equation}
where the sum is now restricted to those subgraphs $H$ for which each nontrivial component intersects $U$ in at most one vertex.
The next lemma follows directly from the proof of Lemma~\ref{lem:decomposition}.

\begin{lemma} \label{lem:block decomposition for 1-mult}
Let $w$ be a $1$-multiplicative graph invariant,
let $G=(V,E)$ be a graph, and let $U\subset V$.
Then for any vertex $v\in V\setminus U$
\begin{equation*}
Z_{\text{block}}(G\mid U;w)=Z_{\text{block}}(G\mid U\cup\{v\};w)+\sum_{\substack{B\in \mathcal{BP}(v,U,G)}} w(B) \cdot Z_{\text{block}}(G\mid U\cup V(B);w).   
\end{equation*}
\end{lemma}
With this lemma, the proof strategy of Theorem~\ref{thm:even zerofree} gives the following result. 
\begin{theorem}\label{thm:block}
Let $w$ be a \emph{$1$-multiplicative} graph invariant and let $G$ be a graph.   
Suppose there exists $a\in (0,1)$ such that for any $U\subset V$ and $v\in V\setminus U$ it holds that
\begin{equation*}
\sum_{\substack{B\in \mathcal{BP}(v,U,G)}} |w(B)|\left(\frac{1}{1-a}\right)^{|V(B)|-2}\leq a.
\end{equation*}
Then $Z_{\text{block}}(G;w)\neq 0$.
\end{theorem}
For convenience of the reader we include a short proof sketch referring to the proof of Theorem~\ref{thm:even zerofree} for some of the steps.
\begin{proof}
The idea is to prove the following statement, which immediately implies the theorem.
For every graph $G=(V,E)$ of maximum degree at most $\Delta$, every $U\subset V$, every $v\in V\setminus U$, we have
\begin{itemize}
    \item[(i)] $Z_{\text{block}}(G\mid U;w)\neq 0$, and
    \item[(ii)] \[
    \left|\frac{Z_{\text{block}}(G\mid U;w)}{Z_{\text{block}}(G\mid U\cup \{v\};w)}-1\right|\leq a. 
    \]
\end{itemize}
We note that (ii) directly implies (i) and so it suffices to show (ii). We will do this by induction on $|V\setminus U|$.
By Lemma~\ref{lem:block decomposition for 1-mult} and induction we have for any vertex $v\notin U$, 
\[
\frac{Z_{\text{block}}(G\mid U;w)}{Z_{\text{block}}(G\mid U\cup \{v\};w)}-1=\sum_{\substack{B\in \mathcal{BP}(v,U,G)}} w(B) \cdot \frac{Z_{\text{block}}(G\mid U\cup V(B);w)}{Z_{\text{block}}(G\mid U\cup \{v\};w)}.
\]
By the same telescoping argument as in the proof of Theorem~\ref{thm:even zerofree} and induction we can then bound
\[
\left|\frac{Z_{\text{block}}(G\mid U;w)}{Z_{\text{block}}(G\mid U\cup \{v\};w)}-1\right|\leq \sum_{\substack{B\in \mathcal{BP}(v,U,G)}} |w(B)|\left(\frac{1}{1-a}\right)^{|V(B)-2}, 
\]
which is bounded by $a$ by the assumption in the theorem.
This concludes the sketch proof.
\end{proof}

Our result gives conditions under which the block polynomial is non-vanishing;
one can think of this result as an analogue of the conditions of Gruber-Kunz~\cites{GK,FerProc}, Koteck\'y-Preiss~\cite{KP86} and Dobrushin~\cite{Dobrushin96} for the non-vanishing of partition functions of polymer models. 
More concretely, for a multiplicative graph invariant $w$ one can define $Z_{\rm pol}(G;w)$ for a graph $G=(V,E)$ (as in~\eqref{eq:def blockpol}) by
\[
Z_{\rm pol}(G;w)=\sum_{H\subseteq E} w(H)=\sum_{H\subseteq E} \prod_{C \text{ comp. of } H} w(C),
\]
where the product runs over the components of $H$.
This can be interpreted as the partition function of a (subset) polymer model with polymers corresponding to connected subgraphs of the base graph $G$.
The Gruber-Kunz conditions state (cf.~\cite{BFP}*{Proposition 3.1}) that if the following conditions hold for each induced subgraph $H=(U,F)$ of $G$ and vertex $v\in U$,
\begin{equation}\label{eq:GK}
\sum_{\substack{B\subseteq F \text{ }v\in V(B)\\B \text{ connected}}} |w(B)|\left(\frac{1}{1-a}\right)^{|V(B)|-1}\leq a,
\end{equation}
for some $a\in (0,1)$, then $Z_{\rm pol}(G;w)\neq 0$.
Theorem~\ref{thm:block} offers two advantages over the condition in~\eqref{eq:GK} in case $w$ is $1$-multiplicative.
First of all, the exponent of $\tfrac{1}{1-a}$ is $|V(B)|-2$ rather than $|V(B)|-1$ and secondly the sum is over block paths $B$ from $v$ to $U$ rather than over all connected subgraphs $B$ that contain the vertex $v$. 
Jackson and Sokal~\cite{JacksonSokal} proved bounds for the block path generating functions in terms of the number of edges that are better than for the connected subgraph generating function.

\section{Concluding remark and questions}
\label{sec:conclusion}
We conclude with a remark and some   questions, starting with an extension to the multivariate case.

\begin{remark}\label{rem:multivariate}
Our results also extend to a multivariate version of the partition function as follows.  For a graph $G=(V,E)$, let $(b_e)_{e\in E}$ be a collection of complex numbers. Define
\[
Z_{\rm Ising}(G;(b_e)):=\sum_{\sigma:V\to \{0,1\}}\prod_{\substack{e=\{u,v\}\in E\\ \sigma(u)=\sigma(v)}}b_e.
\]
Then if for each edge $e$, $|\tfrac{b_e-1}{b_e+1}|\leq (1-\tfrac{1}{\sqrt{2(\Delta-1)}})^2/(\Delta-1)$, then $Z_{\rm Ising}(G;(b_e)) \neq 0$, provided the maximum degree of $G$ does not exceed $\Delta$.
This follows by defining $Z_{\rm even}(G;(x_e))$ in the obvious way and following our proof for the univariate version mutatis mutandis to show that $Z_{\rm even}(G;(x_e))\neq 0$ provided each $x_e$ satisfies $|x_e|\leq (1-\tfrac{1}{\sqrt{2(\Delta-1)}})^2/(\Delta-1)$. Lemma~\ref{lem:vdw} can also easily be adjusted to the multivariate setting. Combining these gives the multivariate result above.
\end{remark}

\begin{question}
We know that the zero-free disk from Theorem~\ref{thm:main} is essentially optimal in one of the real directions (under the assumption that \textsc{P $\neq$ NP}). Is it also optimal in other directions, such as the imaginary direction? This is motivated by connections to quantum computing~\cites{mann,mann2023algorithmic} as mentioned in the introduction.
\end{question}

\begin{question}
Can $(1-\tfrac{1}{\sqrt{2(\Delta-1)}})^2/(\Delta-1)$  be replaced by $1/(\Delta-1)$ in Theorem~\ref{thm:main}? Our current proof of Theorem~\ref{thm:even zerofree} does not appear to leave room for such an improvement, so some new ingredient is required.
\end{question}

\bibliographystyle{alpha}
\bibliography{ising}
\end{document}